\documentclass{amsart}
\usepackage{amscd}
\usepackage{amssymb}
\input xypic
\xyoption{all}
\newdir{ >}{{}*!/-5pt/@{>}}

\def\fp{\mathcal{FP}}

\newcommand{\Z}{\mathbb{Z}}

\numberwithin{equation}{section}
\theoremstyle{plain}
\newtheorem{thm}[equation]{Theorem}
\newtheorem{cor}[equation]{Corollary}
\newtheorem{lem}[equation]{Lemma}
\newtheorem{prop}[equation]{Proposition}

\newcommand{\comment}[1]{}  

\theoremstyle{definition}
\newtheorem{defn}[equation]{Definition}
\theoremstyle{remark}
\newtheorem{rem}[equation]{Remark}
\newtheorem{ex}[equation]{Example}

\begin{document}

\bibliographystyle{plain}

\title{K-theory of $n$-coherent rings}
\author{Eugenia Ellis}
\email{eellis@fing.edu.uy}
\address{IMERL, Facultad de Ingenier\'\i a, Universidad de la Rep\'ublica \\
Julio Herrera y Reissig 565, 11.300, Montevideo, Uruguay}
\author{Rafael Parra}
\email{rparra@fing.edu.uy}
\address{IMERL, Facultad de Ingenier\'\i a, Universidad de la Rep\'ublica \\
Julio Herrera y Reissig 565, 11.300, Montevideo, Uruguay}

\thanks{The first author was partially supported by ANII and CSIC. Both authors were partially supported by PEDECIBA and by the grant ANII FCE-3-2018-1-148588. }



\begin{abstract}
Let $R$ be a strong $n$-coherent ring such that each finitely $n$-presented $R$-module has finite projective dimension. We consider $\mathcal{FP}_{n}(R)$ the full subcategory of $R$-Mod of finitely $n$-presented modules. We prove that $\mathcal{FP}_{n}(R)$ is an exact category,  $K_{i}(R) = K_{i}(\mathcal{FP}_{n}(R))$  for every $i\geq 0$ and we obtain an expression of $\operatorname{Nil}_{i}(R)$.
\end{abstract}

\maketitle
\section{Introduction}
Let $R$ be an associative ring with unit. We consider the following full subcategories of  the category of left $R$-modules denoted by $R$-$\operatorname{Mod}$: 
$$
\begin{array}{lll}
\mathcal{FP}_{0}(R)&=&\{ M \in \mbox{$R$-$\operatorname{Mod}$}: \mbox{ $M$ is finitely generated}\}\\
\mathcal{FP}_{1}(R)&=&\{ M \in \mbox{$R$-$\operatorname{Mod}$}: \mbox{ $M$ is finitely presented}\}\\
\mathcal{FP}_{n}(R)&=&\{ M \in \mbox{$R$-$\operatorname{Mod}$}: \mbox{ $M$ is finitely $n$-presented}\}\\
\mathcal{FP}_{\infty}(R)&=&\{ M \in \mbox{$R$-$\operatorname{Mod}$}: \mbox{ $M$ has a resolution by finitely generated free modules}\}\\
\operatorname{Proj}(R)&=&\{ M \in \mbox{$R$-$\operatorname{Mod}$}: \mbox{ $M$ is finitely generated and projective}\}.\\
\end{array}
$$
We show in Proposition \ref{excat} that $\mathcal{FP}_{n}(R)$ is an exact category for every $n\geq 0$. 
Observe there is a chain of embbeding of exact categories 
$$
\operatorname{Proj}(R)\subseteq \mathcal{FP}_{\infty}(R)\subseteq \mathcal{FP}_{n}(R)\subseteq \ldots \subseteq \mathcal{FP}_{1}(R) \subseteq \mathcal{FP}_{0}(R) \subseteq \mbox{$R$-$\operatorname{Mod}$}.
$$
If $R$ is Noetherian then $\mathcal{FP}_{0}(R)=\mathcal{FP}_{\infty}(R)$ is an abelian category. If $R$ is coherent then $\mathcal{FP}_{1}(R)=\mathcal{FP}_{\infty}(R)$ is an abelian category. In this work we consider rings $R$ such that $\mathcal{FP}_{n}(R)=\mathcal{FP}_{\infty}(R)$ which are called strong $n$-coherent rings. 

We say that $R$ is $n$-regular if the projective dimension of each finitely $n$-presented module is finite. We investigate which techniques used in \cite{Swan} for regular and coherent rings can be applied to $n$-regular and strong $n$-coherent rings. 
In this case we prove in Theorem \ref{fpnco1} that 
$$ K_{i}(R) \simeq K_{i}(\mathcal{FP}_{n}(R)) \qquad{\forall i \geq 0}.$$
The main tool for this result is the Quillen's Resolution Theorem given in \cite{Quillen} applied to the inclusion $\operatorname{Proj}(R)\hookrightarrow \mathcal{FP}_{n}(R)$.
\begin{thm}[Resolution Theorem]\cite[Th 2.1]{Gersten}\label{Qres}
Let $\mathcal{M}$ be an exact category and let $\mathcal{P}$ be a full subcategory closed under extensions in $\mathcal{M}$. Suppose in addition that: 
\begin{enumerate}
\item If $0\rightarrow M \rightarrow P \rightarrow P' \rightarrow 0$ is exact in $\mathcal{M}$ with $P$, $P'$ in $\mathcal{P}$, then $M$ is in $\mathcal{P}$.
\item For every $M$ in $\mathcal{M}$ there is a finite $\mathcal{P}$-resolution of $M$
$$
0\rightarrow P_{k} \rightarrow \ldots \rightarrow P_{1} \xrightarrow{d_{1}} P_{0} \xrightarrow{d_{0}}  M \rightarrow 0
$$
\end{enumerate}
Then the inclusion $\mathcal{P} \hookrightarrow \mathcal{M}$ induces isomorphisms 
$$
K_{m}(\mathcal{P}) \xrightarrow{\simeq} K_{m}(\mathcal{M})
$$
for all $m\geq 0$.
\end{thm}
We can also apply this theorem to the inclusion of exact categories $\operatorname{Nil}(\operatorname{Proj}(R)) \hookrightarrow \operatorname{Nil}(\mathcal{FP}_{n}{(R)})$  which gives us a new expression of $\operatorname{Nil}_{i}(R)$, see  Proposition \ref{losnil}. If $n=1$ Swan in \cite{Swan} proves that the expression obtained is zero, then 
$$K_{i}(R[t])=K_{i}(R) \qquad \mbox{and}  \qquad  K_{i}(R[t,t^{-1}])=K_{i}(R)\oplus K_{i-1}(R)\qquad \mbox{for all $i\geq 1$.} $$ 
This can be done by applying Devissage Theorem to $\mathcal{FP}_{1}(R)\hookrightarrow \operatorname{Nil} (\mathcal{FP}_{1}(R))$ because $\mathcal{FP}_{1}(R)$ is abelian when $R$ is coherent. Let us recall the Quillen's Devissage Theorem:
\begin{thm}[Devissage Theorem] \cite[Th 2.1]{Gersten}\label{Qdev}
If $\mathcal{A}$ is an abelian category and $\mathcal{B}$ is a non-empty full subcategory closed under subobjects, quotient objects, and finite products in $\mathcal{A}$ and if every object $A$ in 
$\mathcal{A}$ has a finite filtration
$$
0=A_{0}\subset A_{1} \subset  \ldots \subset A_{k}=A 
$$
with $A_{i}/A_{i-1}$ in $\mathcal{B}$ for each $i$, then the inclusion $\mathcal{B} \hookrightarrow \mathcal{A}$ induces isomorphisms  
$$
K_{m}(\mathcal{B}) \xrightarrow{\simeq} K_{m}(\mathcal{A})
$$
for all $m\geq 0$.
\end{thm}
Can we apply Devissage Theorem if $n>1$? We can do it if $\mathcal{FP}_{n}(R)$ is abelian but in Proposition \ref{fpab} we prove this only happens when $R$ is coherent. 
The property which arises in $\mathcal{FP}_{n}(R)$ when $R$ is strong $n$-coherent is thickness, $R$ is strong $n$-coherent if and only if $\mathcal{FP}_{n}(R)$ is a thick category.
 
 In Section \ref{sec2} we state the definitions of strong $n$-coherent and $n$-regular ring. We compare this regularity with the $n$-Von Neumann regularity. We show that $\operatorname{Proj}(R)=\fp_{n}(R)$ when $R$ is $n$-Von Neumann regular but this does not happen necessarily  when $R$ is a strong $n$-coherent and $n$-regular ring.  In  Section \ref{sec3} we prove that $\operatorname{Proj}(R)  \hookrightarrow  \fp_{n}(R)$ induces isomorphisms $K_{i}(R)\simeq K_{i}(\fp_{n}(R))$ for all $i\geq 0$ when $R$ is strong $n$-coherent and $n$-regular. In Section \ref{sec4} we generalize \cite[Lemma 6.3]{Swan} in order to apply Resolution Theorem to $\operatorname{Nil}(\operatorname{Proj}(R)) \hookrightarrow \operatorname{Nil}(\fp_{n}(R))$.
 This shows that $\operatorname{Nil}_{i}(R)$ coincides with $\operatorname{Nil}^{n}_{i}(R)$ which is the cokernel of the map $$K_{i}(\fp_{n}(R))\rightarrow K_{i}(\operatorname{Nil}(\fp_{n}(R)))$$
induced by inclusion. In Section \ref{sec5} we prove that we can use Devissage Theorem to prove $\operatorname{Nil}^{n}(R)=0$ if and only if $R$ is coherent, because $\fp_{n}(R)$ is abelian only when $R$ is coherent. We point out that the non coherency of $R$ doesn't imply anything about $\operatorname{Nil}^{n}(R)$.  In the last section we present an application of Theorem \ref{fpnco1} in the particular cases when $R$ is an arithmetic  or a valuation ring.

The projective dimension of an $R$-module $M$ is denoted by $pd(M)$,  the weak dimension of $M$ is denoted by $wd(M)$ and $w.dim(R)$ is the weak dimension of $R$. The expression $R$-module means left $R$-module. 

We would like to thank to Emanuel Rodriguez-Cirone and Maru Sarazola for interesting comments and discussions.   We also thank the referee for his/her comments and corrections.

\section{$n$-coherent modules and rings}\label{sec2}
Let $n\geq0$ be a integer. Recall that an $R$-module $M$ is {\emph{finitely $n$-presented}} if there exists an exact sequence
$$
F_{n}\rightarrow F_{n-1}\rightarrow F_{n-2}\rightarrow \ldots \rightarrow F_{0}\rightarrow M \rightarrow 0
$$
where $F_{i}$ is a finitely generated and free (or projective) module, for every ${0}\leq{i}\leq{n.}$ Following \cite{bp}, we denote by $\fp_{n}(R)$ to the full subcategory of finitely $n$-presented modules. Note that $\fp_{0}(R)$ is the category of finitely generated modules and $\fp_{1}(R)$ is the category of finitely presented modules. Consider the {\emph{$\lambda$-dimension}} of $M$ $$\lambda_{R}(M)=\sup \{ n\geq0: \mbox{$M$ is finitely $n$-presented}\}.$$ If $M$ is not finitely generated we set $\lambda_{R}(M)=-1$. The category formed by modules that posses a resolution by finitely generated free (or projective) modules is $\fp_{\infty}(R)$.  We say that $M$ is \emph{finitely $\infty$-presented} if $M\in\fp_{\infty}(R)$.
We inmediately observe the following  chain of inclusions: 
$$
\fp_{\infty}(R)= \displaystyle \bigcap_{n\geq 0}\fp_{n}(R)\subseteq \ldots \subseteq \fp_{n+1}(R)\subseteq \fp_{n}(R)\subseteq \ldots\fp_{1}(R)\subseteq \fp_{0}(R).
$$ 
The finitely $\infty$-presented modules can be found in the literature as pseudo coherent modules,  see \cite[Example 7.1.4, Ch 2, \textsection 7]{Wei}. 
 The concept of $\lambda$-dimension and the finitely $n$-presented modules are related as follows.
\begin{lem}\label{lemares}\cite[Remark 1.5]{bp} For every $n\geq0$.
\begin{enumerate}
\item \label{lema1} $M\in \fp_{n}(R)$ if and only if $\lambda_{R}(M)\geq{n}.$
\item \label{lema2} $M\in \fp_{n}(R)\backslash \fp_{n+1}(R)$ if and only if $\lambda_{R}(M)=n.$
\item \label{lema3} $M\in \fp_{\infty}(R)$ if and only if $\lambda_{R}(M)={\infty}.$
\end{enumerate}
\end{lem}
The \emph{$\lambda$-dimension of a ring} $R$, denoted by $\lambda(R)$,  was formulated by Vasconcelos in \cite{Vas} in order to study the power series and  polynomial rings. The $\lambda$-dimension of $R$ is the greatest integer $n$ (or ${\infty}$ if none such exits) such that $\lambda_{R}(M)\geq{n}$ implies  $\lambda_{R}(M)={\infty}.$ Recall that R is Noetherian ring if and only if  $\lambda(R)=0$ and R is coherent if and only if $\lambda(R)\leq{1}$.  

As in \cite{ncoh} we  extend the concept of coherent module:

\begin{defn}[\emph{$n$-coherent module}]
Let $n\geq0$ be a integer. An $R$-module $M$ is called $n$-coherent if:
\begin{itemize}
\item $M$ belongs to the category $\fp_{n}(R).$
\item For each submodule $N\hookrightarrow M$ such that $N$ is $(n-1)$-presented then $N$ is also  $n$-presented. 
\end{itemize}
\end{defn}
Denote by \mbox{$n$-$\operatorname{Coh}(R)$} to the collection of all $n$-coherent $R$-modules. Note that if $n=1$ the definition is the same as coherent module.
 
 \begin{rem} A ring $R$ is {\emph{$n$-coherent}} if it is $n$-coherent as an $R$-module with the regular action, (i.e. if each $(n-1)$-presented ideal of $R$ is $n$-presented). We say that $R$ is {\emph{strong $n$-coherent}}  if each $n$-presented $R$-module is $(n+1)$-presented. The idea of $n$-coherent rings has been stated in the literature by D. L. Costa \cite{Costa} but this terminology is not the same as in \cite{ncoh}. A strong $n$-coherence ring is equivalent to a $n$-coherence ring for $n=1$, but it is an open question for $n\geq2$. A coherent ring is a 1-coherent ring (1-strong coherent ring) and 0-strong coherent ring is a Noetherian ring.
\end{rem}

In our definition, $R$ is a strong $n$-coherent ring if and only if  $\lambda(R)\leq{n}$. Denote by $n$-$\operatorname{SCoh}$ the class of all strong $n$-coherent rings. We observe the following chain of inclusions: 
$$
\mbox{$0$-$\operatorname{SCoh}$}\subset \mbox{$1$-$\operatorname{SCoh}$}\subset \mbox{$2$-$\operatorname{SCoh}$}\subset\ldots \subset\mbox{$n$-$\operatorname{SCoh}$}\subset\ldots 
$$

\begin{ex}\cite[Example 1.3]{bp} \label{ejemplo} Let $k$ be a field. Consider the following ring  
$$R=\dfrac{k\left[{x_{1},x_{2},x_{3},\ldots} \right] }{(x_{i}x_{j})_{{i,j}\geq1}}.$$
In \cite{bp} it is shown that $\fp_{2}(R)=\fp_{\infty}(R)$, $(x_{1})\in \fp_{0}(R)\backslash \fp_{1}(R)$ and $R/(x_{1}) \in \fp_{1}(R)\backslash \fp_{2}(R)$.   
We obtain strict inclusions 
$$\fp_{\infty}(R)= \fp_{2}(R) \subset\fp_{1}(R)\subset\fp_{0}(R)$$ which show that $R$ is strong $2$-coherent but is it not coherent. 
\end{ex}

It is well known that $R$ is a coherent ring if and only if the category of coherent modules ($1$-$\operatorname{Coh}(R))$ coincides with the category of finitely presented modules ($\fp_{1}(R)$). 

\begin{prop}\label{fpnco5} Let $n\geq1.$ If $R$ is a strong $n$-coherent ring then
$$
\fp_{n}(R)=\mbox{$n$-$\operatorname{Coh(R)}$. }
$$
\end{prop}
\begin{proof}
Suppose $M\in\fp_{n}(R)$. Then from  \cite[Proposition 1.1]{Zhu} we have a short exact sequence with $K\in\fp_{n-1}(R)$.
$$
0\rightarrow K \rightarrow R^{k}\rightarrow M \rightarrow 0
$$
From our hypothesis we get $R^{k}$ is a $n$-coherent module by \cite[Remark 3.5]{ncoh} and therefore $M\in\mbox{$n$-$\operatorname{Coh(R)}$}$  by \cite[Theorem 2.3]{ncoh}.
\end{proof}

\begin{cor}\cite[Corollary 2.7]{Swan} If $R$ is a coherent ring then $\fp_{1}(R)=\mbox{$1$-$\operatorname{Coh(R)}$.}$
\end{cor}

The category of $\fp_{\infty}(R)$ is thick, see \cite[Theorem 1.8]{bp}. That means it is closed under taking direct summands and for every short exact sequence 
$$
0\rightarrow A\rightarrow B\rightarrow C \rightarrow 0
$$
with two of the three terms $A$, $B$, $C$ in $\fp_{\infty}(R)$ so is the third. 
Because of \cite[Theorem 2.4]{bp}, $R$ is strong $n$-coherent if and only if  $\fp_{n}(R) =\fp_{\infty}(R)$.
\begin{prop}\cite[Corollary 2.6]{bp} 
 Let $n\geq1$ then $R$ is a strong $n$-coherent ring if, and only if, the category $\fp_{n}(R)$ is closed under taking kernels of epimorphims.
\end{prop}
 
\begin{cor}\label{cc1} Let $n\geq1.$ If $R$ is a strong $n$-coherent ring then $n$-$\operatorname{Coh(R)}$ is closed under direct summands, extensions and taking cokernels of monomorphisms and kernels of epimorphism.
\end{cor}

A result of Gersten \cite[Theorem 2.3]{Gersten} shows that if $R$ is a regular coherent ring then the inclusion of the finitely generated projective modules  in the category $\fp_{1}(R)$ induces isomorphisms $K_{i}(\fp_{1}(R))\simeq K_{i}(R)$ for all $i\geq0$. Recall that a ring is said  to be regular if every finitely generated ideal of $R$ has finite projective dimension. If  $R$ is coherent, Quentel shows  in \cite{Quentel} that $R$ is regular if and only if every finitely presented module has finite projective dimension, see \cite[Theorem 6.2.1]{Glaz1}. Motivated by this fact, we introduce the following definition:
  \begin{defn}[\emph{$n$-regular ring}]\label{defnreg}
Let $n\geq1$ be a integer. A ring $R$ is called \emph{$n$-regular} if each finitely $n$-presented $R$-module has finite projective dimension. 
\end{defn}
A $1$-regular ring is a regular ring. In the next section we extend the mentioned result of Gersten to strong $n$-coherent rings which are $n$-regular, see Theorem \ref{fpnco1}. The ideas of strong $n$-coherence and $n$-regularity have been studied before for commutative rings. Following \cite{Costa}, for $n,d\geq0$, recall that a ring is said to be $(n,d)$-Ring if every $n$-presented module has projective dimension at most $d$. These rings are strong $sup\left\lbrace {n,d}\right\rbrace $-coherent. In \cite[Example 6.5]{Costa}  there is an example of a non-coherent ring $R$ of weak dimension one that is a $(2,1)$-Ring.
In the commutative case, a ring is called \emph{$n$-Von Neumann regular} if it is an $(n,0)$-Ring. Thus, 1-Von Neumann regular rings are the Von Neumann regular rings, see \cite[Theorem 1.3]{Costa}. This notion of regularity does not coincide with Definition \ref{defnreg}. Both concepts are related as follows. 
\begin{prop} \cite[Theorem 2.1]{Mahdou}\label{2de3}
Let $R$ be a commutative ring, $R$ is $n$-Von Neumann regular if and only if
\begin{itemize} 
\item[a)] $R$ is strong $n$-coherent, 
\item[b)] $R$ is $n$-regular,
\item[c)] every finitely generated proper ideal of $R$ has non-zero annihilator. 
\end{itemize}
\end{prop}
The rings we are interested in are those with the first two conditions of Proposition \ref{2de3}.
 A ring $R$  such that every finitely generated proper ideal of $R$ has a non-zero annihilator is called $CH$-Ring. Another equivalent definition is as follows: every finitely generated submodule of a projective $R$-module $P$ is a direct summand of $P$ \cite[Theorem 5.4]{Bass}.

Let $n\geq 1$. By  \cite[Theorem 3.9]{Zhu},  every $n$-presented $R$-module is flat if and only if $R$ is $n$-Von Neumann regular. In particular every module of $\mathcal{FP}_{n}(R)$ is projective. Then if $R$ is $n$-Von Neumann regular we obtain 
$$\operatorname{Proj}(R)=\mathcal{FP}_{n}(R).$$
Under the hypothesis of $n$-Von-Neumann regularity there is no sense in studying the relation between the K-theory of $R$ and $\fp_{n}(R)$.

Let $k$ be a field and $E$ be a $k$-vector space with infinite rank. Set $B=k\ltimes E$ the trivial extension of $k$ by $E$. Let $A$ be a Noetherian ring of global dimension 1, and $R=A\times B$ the direct product of $A$ and $B$. By \cite[Theorem 3.4]{Mahdou} $R$ is a $(2,1)$-ring and is not a $2$-Von Neumann regular ring. We obtain an example of a strong 2-coherent and 2-regular ring where 
$$
\operatorname{Proj}(R)\neq \mathcal{FP}_{2}(R).
$$

\section{K-theory of $\fp_{n}(R)$}\label{sec3}

In this section we show $ \fp_{n}(R)$ is an exact category in the sense of \cite[Definition 3.1.1]{ros}.
We also prove that if $R$ is a strong $n$-coherent and $n$-regular ring then $$K_{i}(\fp_{n}(R))\simeq K_{i}(R)  \qquad  \forall i\geq 0.$$
\begin{prop} \label{excat}
$\fp_{n}(R)$ is an exact category. 
\end{prop}
\begin{proof}
The category $\fp_{n}(R)$ is a full additive subcategory of $R$-$\operatorname{Mod}$ which is an abelian category. We have to prove that $\fp_{n}(R)$ is closed under extensions and has a small skeleton. The first assertion follows from \cite[Proposition 1.7]{bp}.   The category of finitely generated $R$-modules $\fp_{0}(R)$  has a small skeleton $\mathcal{S}_{0}$ which is the set of quotient modules of 
$$\{  R^{k}: k\in \mathbb{N}\}. $$
The set $\mathcal{S}_{n}=\fp_{n}(R) \cap \mathcal{S}_{0}$ is a skeleton of $\fp_{n}(R)$.
\end{proof}

\begin{thm}\label{fpnco1}
If $R$ is a strong $n$-coherent and $n$-regular ring then
$$
K_{i}(R)\simeq K_{i}(\mbox{$n$-$\operatorname{Coh(R)}$})\simeq K_{i}(\fp_{n}(R))\qquad \qquad \forall i\geq 0.
$$
\end{thm}
\begin{proof}
We are going to use the Resolution Theorem with $\mathcal{M}= \fp_{n}(R)$  and $\mathcal{P}=\operatorname{Proj}(R)$. 
Note that $\operatorname{Proj}(R)\subseteq \fp_{\infty}(R)\subseteq \fp_{n}(R)$.
Consider the following exact sequence
$$
0\rightarrow P_{1} \rightarrow M \rightarrow P_{2} \rightarrow 0 \qquad \quad \mbox{ with $P_{1},P_{2} \in \operatorname{Proj}(R)$ and $M\in \fp_{n}(R)$.}
$$
Because $P_{2}$ is projective we obtain $M\simeq P_{1} \oplus P_{2}$ then $M$ belongs to $\operatorname{Proj}(R)$. We conclude that $\operatorname{Proj}(R)$ is closed under extension in $\fp_{n}(R)$. 

Let 
$$
0\rightarrow M \rightarrow P_{1} \rightarrow P_{2} \rightarrow 0 \qquad \quad \mbox{ with $P_{1},P_{2} \in \operatorname{Proj}(R)$ and $M\in \fp_{n}(R)$.}
$$
we have $P_{1}\simeq M \oplus P_{2}$ and  $M$ is also projective because $P_{1}$ is projective. 
As $\fp_{\infty}(R)$ is thick then $M$ belongs to $\fp_{\infty}(R)$ and in particular is finitely generated. We conclude that $M\in \operatorname{Proj}(R)$. 

Let $M\in \fp_{n}(R)$. Then there exists a resolution of free modules
$$
F_{n}\xrightarrow{u_{n}} F_{n-1}\rightarrow \ldots \rightarrow F_{1}\rightarrow F_{0} \rightarrow M \rightarrow 0
$$ 
As $R$ is strong $n$-coherent we can take a free resolution as long as we want
$$
\ldots\rightarrow F_{n+1}\rightarrow F_{n}\rightarrow F_{n-1}\rightarrow \ldots \rightarrow F_{1}\rightarrow F_{0} \rightarrow M \rightarrow 0
$$
By hypothesis we know $M$ has finite projective dimension, then there exists $d$   
$$
0\rightarrow \ker u_{d} \rightarrow F_{d} \xrightarrow{u_{d}} \ldots \rightarrow F_{1}\rightarrow F_{0} \rightarrow M \rightarrow 0
$$
with $\ker u_{d}$ projective.  As 
$$
\ker u_{d} = \operatorname{Img} u_{d+1}\simeq F_{d+1}/\ker u_{d+1}
$$
we obtain $\ker u_{d}$ is finitely generated. 
\end{proof}

\begin{cor}\cite[Theorem 2.3]{Gersten} If $R$ is a coherent and regular ring then
$$
K_{i}(R)\simeq K_{i}(\fp_{1}(R))\qquad \qquad \forall i\geq 0.
$$
\end{cor}

\section{The groups $\operatorname{Nil}(\mathcal{FP}_{n})$}\label{sec4}

Let $\mathcal{C}$ be an exact category, $\operatorname{Nil}(\mathcal{C})$ is the category whose objects are the pairs $(A,\alpha)$ where $A\in \mathcal{C}$  and $\alpha: A\rightarrow A$ is a nilpotent endomorphism. A morphism $(A,\alpha)\rightarrow (B,\beta)$ in $\operatorname{Nil}(\mathcal{C})$ is a morphism $f:A\rightarrow B$ in $\mathcal{C}$ such that $f\circ \alpha=\beta\circ f$.
The category $\operatorname{Nil}(\mathcal{C})$ is an exact category and there exist exact functors
$$
\mathcal{C}\rightarrow \operatorname{Nil}(\mathcal{C}) \quad A \mapsto (A,0)\qquad \qquad  \operatorname{Nil}(\mathcal{C}) \rightarrow \mathcal{C} \quad (A,\alpha)
\mapsto A. $$
Taking $\mathcal{C}=\operatorname{Proj}(R)$ and $\mathcal{N}=\operatorname{Nil}(\operatorname{Proj}(R))$, note that $K_{i}(R)$ is a direct summand of $K_{i}(\mathcal{N})$ and define $\operatorname{Nil}_{i}(R)$ the cokernel of $K_{i}(R)\rightarrow K_{i}(\mathcal{N})$. We obtain 
$$
K_{i}(\mathcal{N})= K_{i}(R) \oplus \operatorname{Nil}_{i}(R) \qquad \quad i \geq 0.
$$
 We consider $\mathcal{C}=\fp_{n}(R)$.
 Let ${\mathcal{N}}^{*}_{n}=\operatorname{Nil}(\fp_{n}(R))$ and ${\operatorname{Nil}}^{n}_{i}(R)$ the cokernel of $K_{i}(\fp_{n}(R))\rightarrow K_{i}(\mathcal{N}_{n}^{*})$. We obtain 
$$
K_{i}(\mathcal{N}_{n}^{*})= K_{i}(\fp_{n}(R)) \oplus \operatorname{Nil}^{n}_{i}(R) \qquad \quad i \geq 0.
$$
\begin{lem}\label{lemares}
\begin{enumerate}
\item \label{lema1} \cite[Lemma 6.3]{Swan} For every $(M,\alpha)\in\mathcal{N}^{*}_{1}$ there exist a $(P,\beta)\in \mathcal{N}$ and an epimorphism $\varphi: P\twoheadrightarrow M$ such that the following diagram is commutative
$$
\xymatrix{
P \ar@{->>}[r]^-{\varphi} \ar[d]_-{\beta} &M \ar[d]^-{\alpha} \\
P  \ar@{->>}[r]_-{\varphi} & M
}
$$
\item \label{lema2} If $R$ is coherent and regular then every $(M,\alpha)\in\mathcal{N}^{*}_{1}$ has a finite $\mathcal{N}$-resolution.
\item \label{lema3} If $R$ is strong $n$-coherent and $n$-regular  then every $(M,\alpha)\in\mathcal{N}^{*}_{n}$ has a finite $\mathcal{N}$-resolution.
\end{enumerate}
\end{lem}
\begin{proof}
\begin{enumerate}
\item  Consider $n\in \mathbb{N}$ such that $\alpha^{n+1}=0$. Let $Q$ be a projective and finitely generated module with $f: Q \twoheadrightarrow M$. Let $P=Q^{n+1}$ and define 
$$
\begin{array}{ll}
\beta : P \rightarrow P  & \beta(x_{0},x_{1},\ldots x_{n})=(0,x_{0},\ldots x_{n-1}) \\
\varphi: P\twoheadrightarrow M &\varphi(x_{0},x_{1},\ldots x_{n})= f(x_{0})+\alpha(f(x_{1}))+\alpha^2(f(x_{2}))+\ldots \alpha^n(f(x_{n})).
\end{array}
$$
\item This case is a particular case of \eqref{lema3} with $n=1$.
\item  Given $(M,\alpha)\in\mathcal{N}^{*}_{n} \subseteq \mathcal{N}^{*}_{1}$, by \eqref{lema1} there exist $(P_{0},\beta_{0})\in \mathcal{N} $ and an epimorphism in $\mathcal{N}^{*}_{n}$ such that $\varphi_{0}: (P_{0},\beta_{0}) \twoheadrightarrow (M,\alpha).$
Because $R$ is strong $n$-coherent then $\mathcal{FP}_{n}(R)$ is closed by taking kernels of epimorphisms. Then $(\operatorname{ker}(\varphi_{0}),\beta_{0}|_{\operatorname{ker}(\varphi_{0})})$ belongs to $\mathcal{N}^{*}_{n} \subseteq \mathcal{N}^{*}_{1}$ and we can apply \eqref{lema1} again. 
There exist $(P_{1},\beta_{1})\in \mathcal{N} $ and an epimorphism in $\mathcal{N}^{*}_{n}$ such that $\varphi_{1}: (P_{1},\beta_{1}) \twoheadrightarrow (\operatorname{ker}(\varphi_{0}),\beta_{0}|_{\operatorname{ker}(\varphi_{0})})$. In this way we construct a resolution 
$$
\ldots (P_{i},\beta_{i}) \xrightarrow{\varphi_{i}} (P_{i-1},\beta_{i-1}) \xrightarrow{\varphi_{i-1}}\ldots (P_{0},\beta_{0})\xrightarrow{\varphi_{0}} (M,\alpha) \rightarrow 0
$$ 
As $R$ is $n$-regular, every $M\in \mathcal{FP}_{n}(R)$ has finite projective dimension, then there exists $k\in\mathbb{N}$ such that $P_{k}=0$. We obtain a finite $\mathcal{N}$-resolution of $(M,\alpha)$. 
\end{enumerate}
\end{proof}

\begin{prop}\label{losnil}
If $R$ is a strong $n$-coherent and $n$-regular ring  then $$ \operatorname{Nil}_{i}^{n}(R) \simeq \operatorname{Nil}_{i}(R).$$
\end{prop}
\begin{proof}
We can regard $\mathcal{N}^{*}_{n}$ as the full subcategory of $R[t]$-modules consisting of modules $M$ which are $n$-coherent over $R$ and are such that the action of $t$ is a nilpotent endomorphism 
$$
t: M\rightarrow M \qquad m\mapsto t\cdot m.
$$
We obtain $\mathcal{N}\subseteq \mathcal{N}^{*}_{n}\subseteq R[t]$-$\operatorname{Mod}$. The category $\mathcal{N}^{*}_{n}$ is exact and $\mathcal{N}$ is a full subcategory of $\mathcal{N}^{*}_{n}$ which is closed under kernels of epimorphisms and extensions. As $R$ is strong $n$-coherent and $n$-regular we can use the Lemma \ref{lemares} \eqref{lema3} to obtain that every $M\in \mathcal{N}^{*}_{n}$  has a finite $\mathcal{N}$-resolution.
Because of the Resolution Theorem we obtain that the inclusion $\mathcal{N}\hookrightarrow \mathcal{N}^{*}_{n}$
induces isomorphisms 
$$
K_{i}(\mathcal{N}) \simeq K_{i}(\mathcal{N}^{*}_{n}).
$$
In Theorem \ref{fpnco1} we prove that the inclusion $\operatorname{Proj}(R)\hookrightarrow \mathcal{FP}_{n}(R)$
induces isomorphisms
$$
K_{i}(\operatorname{Proj}(R)) \simeq K_{i}(\mathcal{FP}_{n}(R)).
$$
Then we have the following diagram
$$
\xymatrix{
0\ar[r]  & \operatorname{Nil}_{i} (R)\ar[d]\ar[r] & K_{i}(\mathcal{N})\ar[r] \ar[d]^{\simeq}& K_{i}(\operatorname{Proj}(R))\ar[d]^{\simeq}\ar[r]&0\\ 
0\ar[r]  & \operatorname{Nil}_{i}^{n}(R) \ar[r] &K_{i}(\mathcal{N}^{*}_{n})\ar[r]& K_{i}(\mathcal{FP}_{n}(R))\ar[r] &0
}
$$
Because of the 5-Lemma applied to the previous diagram we obtain 
$$ \operatorname{Nil}_{i}^{n}(R) \simeq \operatorname{Nil}_{i}(R).$$
\end{proof}
\begin{rem}
The previous result with $n=1$ is \cite[Proposition 6.2]{Swan}.
\end{rem}

\section{Categorical properties of $\fp_{n}(R)$}\label{sec5}

The Fundamental Theorem in K-theory relates the K-theory of $R[t]$ and $R[t,t^{-1}]$ with the K-theory of a ring $R$. Let $R$ be a ring then if $i\geq 1$
$$
\begin{array}{c}
K_{i}(R[t])\simeq K_{i}(R) \oplus \operatorname{Nil}_{i-1}(R) \\
K_{i}(R[t,t^{-1}])\simeq K_{i}(R)\oplus K_{i-1}(R) \oplus \operatorname{Nil}_{i-1}(R) \oplus \operatorname{Nil}_{i-1}(R).
\end{array}
$$
We want to know when $\operatorname{Nil}_{i}(R)=0$ in order to obtain an expression which relates the K-groups of $R[t,t^{-1}]$ or $R[t]$ only with the K-groups of $R$. 

If $R$ is coherent it is proved in \cite[Lema 6.4]{Swan} that $\operatorname{Nil}_{i}^{1}(R)=0$ using Devissage Theorem \ref{Qdev}. We are going to check that $\fp_{1}(R)\hookrightarrow\mathcal{N}^{*}_{1 } $ satisfies the hypothesis. The category $\mathcal{N}^{*}_{1}$ is abelian  because $\fp_{1}(R)$ is abelian when $R$ is coherent. We also have $\fp_{1}(R)$ is closed under subobjects, quotient objects and finite products in $\mathcal{N}^{*}_{1}$. For each object $[M,\alpha] \in \mathcal{N}^{*}_{1}$ we can filter $[M,\alpha]$ by $[M_{i},\alpha|_{M_{i}}]$ where $M_{i}=\alpha^{n-i}(M)$ and $n$ is the lowest natural number such that $ \alpha^{n}(M)=0$. The quotient $[M_{i},\alpha_{i}]/[M_{i-1},\alpha_{i-1}]$ is $[M_{i}/M_{i-1},\overline{\alpha}]$. As $\alpha(M_{i})=M_{i-1}$ then $\overline{\alpha}=0$. For this reason $[M_{i},\alpha_{i}]/[M_{i-1},\alpha_{i-1}] \in \fp_{1}(R)$. If $R$ is a regular coherent ring, by Proposition \ref{losnil} we have $\operatorname{Nil}_{i}(R)=\operatorname{Nil}^{1}_{i}(R)$, as we seen above $\operatorname{Nil}_{i}^{1}(R)=0$ and then $\operatorname{Nil}_{i}(R)=0$. It means that 
$$K_{i}(R[t])=K_{i}(R) \qquad \mbox{and}  \qquad  K_{i}(R[t,t^{-1}])=K_{i}(R)\oplus K_{i-1}(R)\qquad \mbox{for all $i\geq1$.} $$ 

If $R$ is a strong $n$-coherent and $n$-regular ring then by Proposition \ref{losnil} we have $\operatorname{Nil}_{i}^{n}(R) \simeq \operatorname{Nil}_{i}(R)$. We can not follow the steps of  the case $n=1$ because $\fp_{n}(R)$ is not necessarily  abelian. We prove in Proposition \ref{fpab}  that this only happens when $R$ is coherent. Although $\fp_{n}(R)$ is not abelian if $R$ is a non-coherent and strong $n$-coherent ring with $n\geq 2$, we have that $\fp_{n}(R)$ is a thick category. 
We expect that  $\operatorname{Nil}_{i}^{n}(R)$ are easier to handle than $\operatorname{Nil}_{i}(R)$.

Recall from \cite{hovey} that a full subcategory $\mathcal{C}$ of $R$-Mod is \emph{wide} if it is abelian and closed under extensions. Let $\mathcal{C}_{0}$ the wide subcategory generated by $R$. Observe $\mathcal{C}_{0}$ contains all finitely presented modules then $\fp_{1}(R)\subseteq \mathcal{C}_{0}$. By \cite[Lema 1.6]{hovey} a ring $R$ is coherent if and only if $\mathcal{C}_{0}=\fp_{1}(R)$.

\begin{prop}\label{fpab}
Let $n\geq1$. The category  $\fp_{n}(R)$ is abelian if, and only if $R$ is a coherent ring. 
\end{prop}
\begin{proof}
If $\fp_{n}(R)$ is abelian then $\mathcal{C}_{0}\subseteq \fp_{n}(R)\subseteq \fp_{1}(R)$. We obtain $\mathcal{C}_{0}= \fp_{1}(R)$ and then $R$ is coherent. 
Reciprocally, if $R$ is coherent then $\fp_{1}(R)$ is an abelian category. By \cite[Theorem 2.4]{bp} we obtain that $\fp_{1}(R)=\fp_{\infty}(R)=\fp_{n}(R)$. Therefore $\fp_{n}(R)$ is an abelian category.
\end{proof}

\section{Arithmetic and valutation rings.}\label{sec6}
In this section R is a commutative ring.
For coherent rings, the regularity condition is related to the weak dimension of $R$-modules. Examples of regular coherent rings includes Von Neumann regular  and semihereditary rings. Coherent rings with finite weak dimension are regular, however the converse is not necessarily true. 

Let us discuss the properties of $R$ depending only on its weak dimension.
If $w.dim(R)=0$ then $R$ is Von Neumann regular which also is coherent. If $w.dim(R)=1$ the coherence is not guaranteed, see \cite[Example 2.3]{Kabbaj}, but in this case $R$ is coherent if and only if $R$ is semihereditary (see \cite[Proposition 2.2]{Glaz2}). Finally  $w.dim(R)\leq1$ if and only if $R$ is an arithmetic reduced ring. 

Recall a ring $R$ is {\emph{arithmetic}} if the ideals of $R_{M}$ are totally ordered by inclusion for all maximal ideals $M$ of $R$. It is proved in \cite[Theorem 2.1]{Couchot} that arithmetic rings are strong 3-coherent and that reduced arithmetic rings are strong 2-coherent. 

Arithmetical condition is not equivalent to strong $n$-coherence. The Example \ref{ejemplo}  is strong $2$-coherent but is not arithmetical neither coherent, see \cite[Example 3.13] {Abuhlail}.

\begin{prop} Let $R$ be an arithmetic ring.
\begin{enumerate}
\item[1.] The ring $R$ is coherent if and only if the  annihilator of every  element is  finitely generated. In particular, if $w.dim(R)<\infty$ then $R$ is a supercoherent regular reduced ring.
\item[2.] If $R$ is 3-regular:
$$
K_{i}(R)\simeq K_{i}(\mbox{$3$-$\operatorname{Coh(R)}$})\simeq K_{i}(\fp_{3}(R))\qquad \qquad \forall i\geq 0.
$$
\item[3.] If $R$ is a reduced ring then: 
$$
K_{i}(R)\simeq K_{i}(\mbox{$2$-$\operatorname{Coh(R)}$})\simeq K_{i}(\fp_{2}(R))\qquad \qquad \forall i\geq 0.
$$
\item[4.]If $R$  has a Krull dimension $0$ and is 2-regular then:  
$$
\left[{Spec(R),\Z}\right] =K_{0}(R)\simeq K_{0}(\mbox{$2$-$\operatorname{Coh(R)}$})\simeq K_{0}(\fp_{2}(R)).
$$
\end{enumerate}
\end{prop}
\begin{proof}
\begin{enumerate}
\item[1.] By \cite[1.4 Fact, Ch XII, \textsection Arithmetic Rings ]{Lombardi} an arithmetic ring is coherent if and
only if the annihilator of every element is finitely generated. An arithmetical coherent ring with $w.dim(R)<\infty$ is a semiheditary ring.
\item[2.]It follows from Theorem \ref{fpnco1} and \cite[Theorem 2.1]{Couchot}.
\item[3.]If $R$ is an arithmetic reduced ring  then $w.dim(R)\leq1$ (see \cite[Theorem 4.2]{GlazRyan},\cite{Jensen}) and $pd(M)\leq1$ for all $M\in\fp_{2}(R)$ \cite[Lemma 8]{cp}.  
\item[4.] By Pierce's Theorem, \cite[Theorem 2.2.2]{Wei}, $\left[{Spec(R),\Z}\right] =K_{0}(R)$ for every $R$ with Krull dimension $0$. By \cite[Corollary 2.7]{Couchot} $R$ is a strong $2$-coherent ring.
\end{enumerate}
\end{proof}

Recall $R$ is arithmetic if $R$ is locally a valuation ring. 
A ring $R$ is a {\emph{valuation ring}} if  the set of ideals of $R$ is totally ordered by inclusion.

\begin{prop} Let $R$ be a valuation ring, $M$ its maximal ideal and $Z$ the subset of its zero divisors. 

\begin{enumerate}
\item[1.] If $Z=0$ then  $R$ is a coherent ring. 

\item[2.]If $Z\neq 0$, $Z\neq{M}$ and $R$ is $2$-regular then: 
$$
K_{i}(R)\simeq K_{i}(\mbox{$2$-$\operatorname{Coh(R)}$})\simeq K_{i}(\fp_{2}(R))\qquad \qquad \forall i\geq 0.
$$
$\fp_{\infty}(R) =\fp_{2}(R)\varsubsetneq\fp_{1}(R)$.
\item[3.] If $Z\neq 0$, $Z={M}$ and $R$ is $2$-regular then: 
$$
K_{i}(R)\simeq K_{i}(\mbox{$2$-$\operatorname{Coh(R)}$})\simeq K_{i}(\fp_{2}(R))\qquad \qquad \forall i\geq 0.
$$
$\fp_{\infty}(R)=\fp_{2}(R)\subseteq \fp_{1}(R)$.

\end{enumerate}
\end{prop}
\begin{proof}
Straightforward from Theorem \ref{fpnco1} and \cite[Theorem 2.11]{Couchot}. 
\end{proof}

\bibliographystyle{plain}

\end{document}